\documentclass[oneside,english]{amsart}
\usepackage[T1]{fontenc}
\usepackage[latin9]{inputenc}
\usepackage{amstext}
\usepackage{amsthm}
\usepackage{amssymb}
\usepackage{graphicx}
\usepackage{setspace}
\usepackage{wasysym}

\usepackage{tikz}
\usetikzlibrary{arrows,calc}
\tikzset{
	>=stealth',
	help lines/.style={dashed, thick}, axis/.style={<->}, important
	line/.style={thick}, connection/.style={thick, dotted}, }

\makeatletter
\numberwithin{equation}{section}
\numberwithin{figure}{section}
\theoremstyle{plain}
\ifx\thechapter\undefined
\newtheorem{thm}{\protect\theoremname}[section]
\else
\newtheorem{thm}{\protect\theoremname}[section]
\fi
\theoremstyle{definition}
\newtheorem{defn}[thm]{\protect\definitionname}
\theoremstyle{definition}
\newtheorem*{defn*}{\protect\definitionname}
\theoremstyle{plain}
\newtheorem{lem}[thm]{\protect\lemmaname}
\theoremstyle{plain}
\newtheorem*{thm*}{\protect\theoremname}
\theoremstyle{remark}
\newtheorem{rem}[thm]{\protect\remarkname}
\theoremstyle{plain}
\newtheorem{conjecture}[thm]{\protect\conjecturename}
\theoremstyle{plain}
\newtheorem{prop}[thm]{\protect\propositionname}
\theoremstyle{remark}
\newtheorem{claim}[thm]{\protect\claimname}
\theoremstyle{plain}
  \newtheorem{cor}[thm]{\protect\corollaryname}

\makeatother

\usepackage{babel}
\providecommand{\corollaryname}{Corollary}
\providecommand{\claimname}{Claim}
\providecommand{\conjecturename}{Conjecture}
\providecommand{\definitionname}{Definition}
\providecommand{\lemmaname}{Lemma}
\providecommand{\propositionname}{Proposition}
\providecommand{\remarkname}{Remark}
\providecommand{\theoremname}{Theorem}
\providecommand{\theoremname}{Theorem}

\def\qed{\hfill $\vcenter{\hrule height .3mm
		\hbox {\vrule width .3mm height 2.1mm \kern 2mm \vrule width .3mm
			height 2.1mm} \hrule height .3mm}$ \bigskip}

\def\to{\rightarrow}

\def\RR{\mathbb{R}}

\def\vol{{\rm Vol}}

\def\cal{\mathcal}

\begin{document}
	
	\title{Isotropic Measures and Maximizing Ellipsoids: Between John and Loewner}
	
	\author{Shiri Artstein-Avidan, David Katzin}\thanks{Authors were supported in part by ISF grant No. 665/15}

	\address{School of Mathematical Sciences, Tel Aviv University, Tel Aviv 69978, Israel}
	\date{December 4, 2016}
	\email{shiri@post.tau.ac.il, david.katzin@wur.nl}
	\subjclass[2010]{52A40, 52A05, 28A75}
	\keywords{John position; John ellipsoid; M-position; Isotropic measures; Maximal intersection position}
	
	\maketitle
	\begin{abstract}
We define a one parameter family of positions of a convex body which interpolates
	between the John position and the Loewner position: for $r>0$, we say that $K$ is in {\em maximal intersection position of radius $r$} if $\vol_{n}(K\cap rB_{2}^{n})\geq \vol_{n}(K\cap rTB_{2}^{n})$ for all $T\in SL_{n}$. 
We show that under mild
 	conditions on $K$,  each such
 	position induces a corresponding isotropic measure on the sphere,
 	which is simply a normalized Lebesgue measure on $r^{-1}K\cap S^{n-1}$. In particular, for $r_{M}$ satisfying $r_{M}^{n}\kappa_{n}=\vol_{n}(K)$, the maximal intersection
 position of radius $r_{M}$ is an $M$-position, so we get an $M$-position with
 	an associated isotropic measure.
Lastly, we give an interpretation of John's theorem on contact points as a limit case of 
the measures induced from the maximal intersection positions. 
	\end{abstract}
	
	\maketitle
	
	\section{Introduction and main results\label{sec:Introduction}}
	
	Given a convex body (that is, a compact convex set with non-empty
	interior) in $\mathbb{R}^{n}$, the \emph{John ellipsoid }$J(K)$
	is the maximum-volume ellipsoid contained in $K$. The body $K$ is
	in \emph{John position }if $J(K)=B_{2}^{n}$, the Euclidean unit ball.
	Dually, the \emph{Loewner ellipsoid }$L(K)$ is the minimum-volume
	ellipsoid containing $K$, and $K$ is in \emph{Loewner position }if
	$L(K)=B_{2}^{n}$. The John and Loewner positions always exist and
	are unique up to orthogonal transformations. They are dual in the
	sense that $J(K^{\circ})=(L(K))^{\circ}$ where $L^{\circ}=\left\{ y:\left\langle x,y\right\rangle \leq1\text{ }\forall x\in L\right\} $
	is the \emph{dual body} of $L$ (see \cite{Artstein-Avidan2015} for more details). 
	
	A finite Borel measure $\mu$ on $S^{n-1}$ is \emph{isotropic} if 
	\[\int_{S^{n-1}}\left\langle x,\theta\right\rangle ^{2}d\mu(x)=\frac{\mu(S^{n-1})}{n} \]
for all $\theta\in S^{n-1}$. In 1948, Fritz John \cite{john1948extremum}
	showed the following:
	\begin{thm}[John]
		\label{thm:john}Let $K\subset\mathbb{R}^{n}$ be a convex body in
		John position. Then there exists an isotropic measure whose support
		is contained in $\partial K\cap S^{n-1}$. Moreover, there exists
		such a measure whose support is at most $n(n+1)/2$ points. 
	\end{thm}
	\noindent A reverse result was given by K.~Ball \cite{Ball1992}, who showed that if $B_2^n \subseteq K$ and there is  an isotropic measure supported on $\partial K\cap S^{n-1}$, then $K$ is in John position. 
	By duality, the same result holds for a body in Loewner position.
	
	John's theorem is a special case of a general phenomenon: the family
	$\left\{ TK:T\in SL_{n}\right\}$ of a convex body $K$ is called
	the family of \emph{positions} of $K$. Giannopoulos and Milman \cite{Giannopoulos2000}
	showed that solutions to extremal problems over the positions of a
	convex body often give rise to isotropic measures, and demonstrated this fact for, among others, the John position, the isotropic
	position, the minimal surface area position, and an $M$-position.
	
	In this work, we consider a one-parametric family of extremal positions
	which seems not to have been considered before:
	\begin{defn}
		\label{def:max-intersection-ellipsoid}For a centrally symmetric convex
		body $K\subset\mathbb{R}^{n}$, the ellipsoid $\mathcal{E}_{r}$ of
		volume $r^{n}\kappa_{n}$ is a \emph{maximum intersection ellipsoid
			of radius $r$}, if 
		\[
		\vol_{n}\left(K\cap\mathcal{E}_{r}\right)\geq\vol_{n}\left(K\cap\mathcal{E}\right)
		\]
		for all ellipsoids $\mathcal{E}$ of volume $r^{n}\kappa_{n}$, where
		$\kappa_{n}=\vol_{n}(B_{2}^{n})$. We say that $K$ is in \emph{maximal
			intersection position of radius $r$ }if $rB_{2}^{n}$ is a maximum
		intersection ellipsoid of radius $r$.
	\end{defn}

	In the following, $\mathcal{E}_{r}$ will always denote a maximum intersection ellipsoid of radius $r$. The set of maximal intersection positions interpolates between the
	John and Loewner positions: indeed, let $r_{J}$ be a positive number
	satisfying $\vol_{n}(J(K))=r_{J}^{n}\kappa_{n}$, and let $r_{L}$
	be such that $\vol_{n}(L(K))=r_{L}^{n}\kappa_{n}$. It can be
	easily shown that $K$ is in maximal intersection position of radius
	$r_{J}$ if and only if $r_{J}^{-1}K$ is in John position, and similarly
	for the Loewner position. In other words, up to a scaling, the maximal
	intersection position of radius $r_{J}$ is the John position, and
	the maximal intersection position of radius $r_{L}$ is the Loewner
	position.
	
	Our first result is the following: 
	\begin{thm}
		\label{thm:max-intersection-isotropic} Let $K\subset\mathbb{R}^{n}$
		be a centrally symmetric convex body such that $\vol_{n-1}(\partial K\cap\partial\mathcal{E})=0$
		for all but finitely many ellipsoids $\mathcal{E}$, $\vol_{n-1}(\partial K\cap rS^{n-1})=0$,
		and $\vol_{n-1}(K\cap rS^{n-1})>0$. If $K$ is in maximal intersection
		position of radius $r$, then the restriction of the surface area
		measure on the sphere to $S^{n-1}\cap r^{-1}K$ is an isotropic measure.
	\end{thm}

	\begin{rem}\label{remark-thm-derivate}
		Note that the condition $\mbox{Vol}_{n-1}(\partial K\cap rS^{n-1})=0$
		cannot be omitted. As an example, consider the convex hull of a ball
		and two points, e.g., $K=\text{conv}\{B_{2}^{2}\cup{(\pm\sqrt{2},0)}\}\subset\mathbb{R}^{2}$.
		Here one may check that $K$ is in John position, and so it is in maximal intersection
		position of radius 1. However, the restriction of the surface area measure to $K\cap S^{n-1}$
		is clearly not isotropic, as it has more weight in the direction of the $y$ axis than in the direction of the $x$ axis.		
	\end{rem}
	
	We will denote the surface area measure on the sphere by $\sigma$, and for a Borel set $A \subset \RR^n$ with $\sigma (A\cap S^{n-1}) >0$ we let $\mu_{A}$ denote the restriction of $\sigma$ to  $A$, i.e. 
		\[
		\mu_{A}(B)=\frac{\sigma(B\cap A \cap S^{n-1})}{\sigma(A\cap S^{n-1})}.
		\]
		
 Note that if $\mu_{A}$ is isotropic and $\sigma (S^{n-1}\backslash A)>0$, then $\mu_{S^{n-1}\backslash A}$ is also isotropic.

	Theorem \ref{thm:max-intersection-isotropic} shows that as in \cite{Giannopoulos2000}, an extremal position induces
	an isotropic measure. Contrary to John's  Theorem \ref{thm:john},
	in our case we have an explicit description of the isotropic measure, which is uniform on $r^{-1}K \cap S^{n-1}$, 
	namely it is $\mu_{r^{-1}K}$.
	
	Theorem \ref{thm:max-intersection-isotropic} does not formally include the result of Theorem \ref{thm:john}, in the case $r=r_J=1$, since  
for $K$ in John position we have $S^{n-1}\subset K$, so  Theorem \ref{thm:max-intersection-isotropic} merely states that $\sigma$ is isotropic, a triviality. 	
	Nevertheless, our second result gives a new interpretation to John's Theorem. We show that when $K$ is in John position, the isotropic
	measure which is guaranteed to exist by Theorem \ref{thm:john} may
	be constructed as a limit of the isotropic measures from Theorem \ref{thm:max-intersection-isotropic}.
	In other words, as $r$ approaches $r_{J}$, the corresponding induced
	measures approach a measure of the type described in John's theorem:
	\begin{thm}
		\label{thm:measure-limit-john}Let $K\subset\mathbb{R}^{n}$ be a
		centrally symmetric convex body in John position such that $\vol_{n-1}(\partial K\cap\partial\mathcal{E})=0$
		for all but finitely many ellipsoids $\mathcal{E}$. For every $r>1$,
		denote by $\mu_{r}$ the uniform probability measure on $S^{n-1}\backslash r^{-1}T_{r}K$,
		where $T_{r}K$ is in maximal intersection position of radius $r$.
		Then there exists a sequence $r_{j}\searrow1$ such that the sequence
		of measures $\mu_{r_{j}}$ weakly converges to an isotropic measure
		whose support is contained in $\partial K\cap S^{n-1}$.
	\end{thm}
	A similar result holds for the Loewner position:
	\begin{thm}
		\label{thm:measure-limit-loewner}Let $K\subset\mathbb{R}^{n}$ be
		a centrally symmetric convex body in Loewner position such that $\vol_{n-1}(\partial K\cap\partial\mathcal{E})=0$
		for all but finitely many ellipsoids $\mathcal{E}$. For every $r<1$,
		denote by $\nu_{r}$ the uniform probability measure on $S^{n-1}\cap r^{-1}T_{r}K$,
		where $T_{r}K$ is in maximal intersection position of radius $r$.
		Then there exists a sequence $r_{j}\nearrow1$ such that the sequence
		of measures $\nu_{r_{j}}$ weakly converges to an isotropic measure
		whose support is contained in $\partial K\cap S^{n-1}$.
	\end{thm}
	In the range $[r_{J},r_{L}]$ there is a special radius which we denote
	$r_{M}$, defined so that $\vol_{n}(K)=r_{M}^{n}\kappa_{n}$,
	and for this special radius the maximal intersection position of radius
	$r_{M}$ is an $M$-position. To explain what this means we need a
	few more definitions and background.
	
	In the mid-80s, Vitali Milman \cite{milman1986inegalite} discovered
	the existence of a position for convex bodies which enabled him, and
	the researchers following, to prove many new results, and had
	a major influence on the field. This position, now called $M$-position, can be described in many different and equivalent ways.
	We choose one such way, and for an extensive description and the many
	equivalences see \cite{Artstein-Avidan2015}.
	\begin{thm}[Milman]
		\label{thm:M-pos} There exists a universal constant $C>0$ such
		that for every $n\in\mathbb{N}$ and any centrally symmetric convex
		body $K\subset\mathbb{R}^{n}$, there exists a centrally symmetric
		ellipsoid $\mathcal{E}$ with $\vol_{n}(\mathcal{E})=\vol_{n}(K)$
		such that 
		\begin{equation}
		\frac{\vol_{n}(K^{\circ}+\mathcal{E^{\circ}})}{\vol_{n}(K^{\circ}\cap\mathcal{E^{\circ}})}\frac{\vol_{n}(K+\mathcal{E})}{\vol_{n}(K\cap\mathcal{E})}\leq C^{n}.\label{eq-M-position}
		\end{equation}
	\end{thm}
	In fact, one may show that if an ellipsoid of the same volume as $K$
	satisfies any of the four inequalities 
	\begin{eqnarray*}
		\vol(K^{\circ}+\mathcal{E^{\circ}})\leq c_{1}^{n}\vol_{n}(K), &  & \vol_{n}(K^{\circ}\cap\mathcal{E^{\circ}})\geq c_{1}^{-n}\vol_{n}(K),\\
		\vol(K+\mathcal{E})\leq c_{1}^{n}\vol_{n}(K), &  & \vol_{n}(K\cap\mathcal{E})\geq c_{1}^{-n}\vol_{n}(K),
	\end{eqnarray*}
	then it must satisfy inequality $\eqref{eq-M-position}$ with some
	constant $C=C(c_{1})$ depending only on $c_{1}$ and not on the body $K$
	or on the dimension. For this reason, we shall use the following simple
	definition for $M$-position: 
	\begin{defn}
		A centrally symmetric convex body $K$ is in $M$\emph{-position}
		with constant $C$ if the centrally symmetric Euclidean ball of radius
		$\lambda=\left(\frac{\vol(K)}{\kappa_{n}}\right)^{1/n}$ satisfies
		\[
		\vol_{n}(K\cap\lambda B_{2}^{n})\geq C^{-n}\vol_{n}(K).
		\]
		Since Milman's theorem implies that there exists some universal $C$
		for which any body has an affine image in $M$-position with constant
		$C$, we shall usually omit the words ``with constant $C$'' and
		talk simply of ``$M$-position'', by which we mean an $M$-position
		with respect to the constant $C$ guaranteed by Milman's Theorem 
		\ref{thm:M-pos}.
	\end{defn}
	Clearly, when we maximize the volume of the intersection of K and an ellipsoid
	of volume $\vol_{n}(K)$, we get an $M$-ellipsoid, and when
	it is a Euclidean ball we get that $K$ is in $M$-position. We have then:
\begin{cor}\label{cor-m-position}
	Let $K\subset\mathbb{R}^{n}$
	be a centrally symmetric convex body such that $\vol_{n-1}(\partial K\cap\partial\mathcal{E})=0$
	for all but finitely many ellipsoids $\mathcal{E}$, $\vol_{n-1}(\partial K\cap r_{M}S^{n-1})=0$,
	and $\vol_{n-1}(K\cap r_{M}S^{n-1})>0$, where $r_{M}=\left(\frac{\vol(K)}{\kappa_{n}}\right)^{1/n}.$ If $K$ is in maximal intersection
	position of radius $r_{M}$, then $K$ is in $M$-position, and the restriction of the surface area
	measure on the sphere to $S^{n-1}\cap r_{M}^{-1}K$ is an isotropic measure.
	\end{cor}
		
	This paper is organized as follows: in Section \ref{sec-prelim} we provide 
	some basic results regarding the maximal intersection position. The section concludes with a detailed proof of the main ingredient for the proof of Theorem \ref{thm:max-intersection-isotropic}. In Section \ref{sec-main-proofs} we prove the main theorems \ref{thm:max-intersection-isotropic}, \ref{thm:measure-limit-john}, and \ref{thm:measure-limit-loewner}. The last section discusses the question of uniqueness of the maximum intersection position, a question that is still open. We show that uniqueness follows from a variant of the $(B)$ conjecture.
	
	\section{Preliminaries}\label{sec-prelim}
In this section we provide some results needed for the proof of the main theorems. We start by showing that for $r>0$, the maximal intersection position
	of radius $r$ does in fact exist. We will make frequent use of the
	following function:
	\begin{defn}
		\label{def:m-function}For a centrally symmetric convex body $K=-K\subset\mathbb{R}^{n}$,
		define for every $r>0$, 
		\begin{equation}
		m(r)=\sup\left\{ \vol_{n}(K\cap\mathcal{E}):\mathcal{E}\text{ is an ellipsoid of volume }r^{n}\kappa_{n}\right\} .\label{eq:m-function}
		\end{equation}
	\end{defn}
	Our first lemma shows that a maximal intersection ellipsoid always
	exists:
	\begin{lem}
		\label{lem:max-ellipsoid-exists}For every centrally symmetric convex
		body $K\subset\mathbb{R}^{n}$ and every $r>0$, the supremum in 
		(\ref{eq:m-function}) is attained.
	\end{lem}
	\begin{proof}
		First note that since $K=-K$, the Brunn-Minkowski inequality implies that for every $x\in\mathbb{R}^{n}$
		and every $T\in SL_{n}$, we have 
		\begin{equation}
		\vol_{n}\left(K\cap\left(TB_{2}^{n}+x\right)\right)\leq\vol_{n}\left(K\cap TB_{2}^{n}\right)\label{eq:BM-symmetric-ellipsoid}
		\end{equation}
		and so if the supremum is attained, it is attained on a centrally
		symmetric ellispoid. Note that the supremum may also be attained on a
		non-centrally symmetric ellipsoid only if we have equality in (\ref{eq:BM-symmetric-ellipsoid}),
		which is only possible if $K\cap\left(TB_{2}^{n}+x\right)$ and $K\cap\left(TB_{2}^{n}-x\right)$
		are homothetic. This occurs, for instance, in the case $\left(TB_{2}^{n}+x\right)\subset K$
		or $K\subset\left(TB_{2}^{n}+x\right)$, i.e., when $r<r_{J}$ or
		$r>r_{L}$.
		
		Let $\mathcal{E}_{j}=T_{j}B_{2}^{n}$ be a sequence of centrally symmetric
		ellipsoids where $T_{j}$ is positive definite with $\text{det}T_{j}=r^{n}$
		and $\vol_{n}(K\cap T_{j}B_{2}^{n})\rightarrow m(r)$. If the
		sequence defined by the maximum eigenvalue of $T_{j}$ grows to infinity
		then $\vol_{n}(K\cap T_{j}B_{2}^{n})\to0\neq m(r)$, so the
		set of eigenvalues of $\{T_{j}\}_{j=1}^{\infty}$ must be bounded,
		which implies that the ellipsoids $T_{j}B_{2}^{n}$ are all contained
		in a compact set. It now follows from Blaschke's selection theorem
		that there exists a subseqeunce of ellipsoids converging in the Hausdorff
		distance to a centrally symmetric ellipsoid $\mathcal{E}$ of volume
		$r^{n}\kappa_{n}$ with $\vol_{n}(K\cap\mathcal{E})=m(r)$.
	\end{proof}
	
	Note the following properties of $m(r)$:
	\begin{lem}\label{lem-m(r)-continuous}
		Let $K\subset\mathbb{R}^n$ be a centrally symmetric convex body. We have that \\
		(1) For $0<r\leq r_{J}$ we have $m(r)=r^{n}\kappa_{n}$ and for
		$r\geq r_{L}$ we have $m(r)=\vol_{n}(K)$. \\
		(2) The function $m(r)$ is strictly monotone increasing in $[r_{J},r_{L}]$.\\
		(3) $m(r)$ is continuous, and satisfies for $t\leq s$ that 
		\[
		m(t)\le m(s)\le\left(\frac{s}{t}\right)^{n}m(t).
		\]
	\end{lem}
	\begin{proof}
		Fact (1) is trivial. For (2) let $r_{J}\leq t<s\leq r_{L}$ and choose
		some intersection maximizing ellipsoid $\mathcal{E}_{t}$. Then 
		\[
		m(t)=\vol_{n}(K\cap\mathcal{E}_{t})\leq\vol_{n}\left(K\cap\frac{s}{t}\mathcal{E}_{t}\right)\leq\vol_{n}\left(K\cap\mathcal{E}_{s}\right).
		\]
		
		If the last inequality is an equality then $K\cap\mathcal{E}_{t}=K\cap\frac{s}{t}\mathcal{E}_{t}$
		which is only possible if $K\subset\mathcal{E}_{t}$ (which is impossible
		since $t<r_{L}$) or if $\frac{s}{t}\mathcal{E}_{t}\subset K$ (which
		is impossible since $s>r_{J}$). \\
		To prove (3) it is enough to show the right hand side inequality
		and to this end simply note that 
		\begin{eqnarray*}
			m(t) & = & \vol_{n}(K\cap\mathcal{E}_{t})\geq\vol_{n}\left(K\cap\frac{t}{s}\mathcal{E}_{s}\right)\\
			& \geq & \vol_{n}\left(\frac{t}{s}K\cap\frac{t}{s}\mathcal{E}_{s}\right)=\left(\frac{t}{s}\right)^{n}\vol_{n}\left(K\cap\mathcal{E}_{s}\right)=\left(\frac{t}{s}\right)^{n}m(s).
		\end{eqnarray*}
	\end{proof}
	
	By continuity of $m(r)$, we have:
	
	\begin{lem}
		\label{lem-convergence-of-ellipse-toJohn} Let $K\subset\mathbb{R}^n$ be a centrally symmetric convex body. As $r\searrow r_{J}$ the ellipsoids ${\cal E}_{r}$ converge to $\mathcal{E}_{r_{J}}=J(K)$
		in the Hausdorff distance. 
	\end{lem}
	\begin{proof}
		Since $\vol_{n}(K\cap J(K))=\vol_{n}(J(K))$ then by the continuity
		of $m(r)$, both $\vol_{n}(K\cap\mathcal{E}_{r})$ and $\vol_{n}(\mathcal{E}_{r})$
		approach $m(r_{J})=r_{J}^{n}\kappa_{n}$ as $r\searrow r_{J}$. Let
		$T_r$ be a sequence of transformations such that $T_r\mathcal{E}_r=B_{2}^{n}$.
		As before, since ${\vol_{n}(K\cap T_{r}^{-1}B_{2}^{n})}\to m(r_J)$ then the set $\mathcal{E}_{r}$
		is contained in a compact set. We thus have a converging subsequence $\mathcal{E}_{r_{j}}\rightarrow\mathcal{E}$
		with $\vol_{n}(\mathcal{E})=\vol_{n}(K\cap\mathcal{E})=r_{J}^{n}\kappa_{n}$,
		so $\mathcal{E}$ is an ellipsoid contained in $K$ with the same
		volume as $J(K)$, which is unique. It follows that $\mathcal{E}=J(K)$.
		Since this was true for any converging subsequence, we get that $\mathcal{E}_{r}$
		converges to $J(K)$ as $r\searrow r_{J}$. 
	\end{proof}
	
	We will make use of the following fact. The proof is a simple exercise, see e.g. \cite{Artstein-Avidan2015}:
	\begin{lem}\label{lem:isotropic-traceless-matrix}
		A Borel measure $\mu$ on $S^{n-1}$ is isotropic if and only if every $A\in M_n(\mathbb{R})$ such that $\textrm{tr}A=0$ has 
		\[
		\int_{S^{n-1}}\left\langle x,Ax\right\rangle d\mu (x) = 0.
		\]
	\end{lem}
	
	Lastly, the following theorem is essential for the proof of Theorem \ref{thm:max-intersection-isotropic}:
	\begin{thm}
		\label{thm:derivative-of-volume}Let $K\subset\mathbb{R}^{n}$ be
		a centrally symmetric convex body such that $\mbox{Vol}_{n-1}(\partial K\cap\partial\mathcal{E})=0$
		for all but finitely many ellipsoids $\mathcal{E}$, $\mbox{Vol}_{n-1}(\partial K\cap S^{n-1})=0$
		and $\mbox{Vol}_{n-1}(K\cap S^{n-1})>0$. Let $A\in M_{n}(\mathbb{R})$
		with $\mbox{tr}A=0$, and let $V(t):\mathbb{R}\rightarrow\mathbb{R}$
		be defined by $V(t)=\mbox{Vol}_{n}(K\cap e^{tA}B_{2}^{n})$ . If $K$
		is in maximal intersection position of radius 1, then 
		\[
		\left.\frac{dV(t)}{dt}\right|_{t=0}=\int_{S^{n-1}\cap K}\left\langle x,Ax\right\rangle dS(x)
		\]
		where $S=\vol_{n-1}$ is the surface area measure.
	\end{thm}
	
	We will see in the next section that Theorem \ref{thm:max-intersection-isotropic} is almost a direct corollary of Theorem \ref{thm:derivative-of-volume}. However, Remark \ref{remark-thm-derivate} shows that some caution is needed, and especially, the use of the assumption $\vol_{n-1}(\partial K \cap S^{n-1})=0$ should be identified. Therefore, while the following proof is basically a direct application of some fundamental results in calculus, we provide full details.

\begin{proof}[Proof of Theorem \ref{thm:derivative-of-volume}]

	Let $\left\{ \phi_{j}\right\} _{j=1}^{\infty}$ be a sequence of continuous
	functions from $\mathbb{R}^{n}$ to $\mathbb{R}$ approximating $\mathbf{1}_{\mbox{int}B_{2}^{n}}$,
	chosen as 
	\[
	\phi_{j}(x)=\begin{cases}
	1 & |x|\leq1-\frac{1}{j}\\
	g_{j}(x) & 1-\frac{1}{j}\leq|x|\leq1\\
	0 & |x|\geq1
	\end{cases}
	\]
	where $g_{j}(x):\mathbb{R}^{n}\rightarrow[0,1]$ is chosen such that
	$\phi_{j}(x)$ is continuously differentiable and there is a constant
	$c$ such that $0<|\nabla\phi_{j}(x)|<jc$ for all $x$. For instance
	we may take $g_{j}(x)=\frac{1}{2}-\frac{1}{2}\cos j\pi(|x|-1)$ to
	have $\nabla g_{j}(x)=\frac{j\pi x}{2|x|}\sin j\pi\left(|x|-1\right)$
	.
	
	Similarly, let $\psi_{j}(x)$ be a family of functions approximating
	$\mathbf{1}_{\mbox{int}K}$, chosen as
	
	\[
	\psi_{j}(x)=\begin{cases}
	1 & \left\Vert x\right\Vert _{K}\leq1-\frac{1}{\sqrt{j}}\\
	h_{j}(x) & 1-\frac{1}{\sqrt{j}}\leq\left\Vert x\right\Vert _{K}\leq1\\
	0 & \left\Vert x\right\Vert _{K}\geq1
	\end{cases}
	\]
	where $h_{j}(x):\mathbb{R}^{n}\rightarrow[0,1]$, $\psi(x)$ is continuously
	differentiable, and $0<|\nabla\psi_{j}(x)|<c\sqrt{j}$ for all $x$.
	
	As $j\rightarrow\infty$, $\phi_{j}(x)$ converges pointwise to $\mathbf{1}_{\mbox{int}B_{2}^{n}}$
	and $\psi_{j}(x)$ converges pointwise to $\mathbf{1}_{\mbox{int}K}$.
	We have then: 
	\begin{eqnarray*}
		\left.\frac{d}{dt}\right|_{t=0}V(t) & = & \left.\frac{d}{dt}\right|_{t=0}\int_{\mathbb{R}^{n}}\mathbf{1}_{\mbox{int}B^{n}}\left(e^{-tA}x\right)\mathbf{1}_{\mbox{int}K}(x)\\
		& = & \left.\frac{d}{dt}\right|_{t=0}\int_{\mathbb{R}^{n}}\underset{j\rightarrow\infty}{\lim}\phi_{j}\left(e^{-tA}x\right)\psi_{j}(x)dx.
	\end{eqnarray*}
	
	We will show that the following hold in a neighborhood of $t=0$:
	\begin{eqnarray}
	\frac{d}{dt}\int_{\mathbb{R}^{n}}\underset{j\rightarrow\infty}{\lim}\phi_{j}\left(e^{-tA}x\right)\psi_{j}(x)dx=\frac{d}{dt}\underset{j\rightarrow\infty}{\lim}\int_{\mathbb{R}^{n}}\phi_{j}\left(e^{-tA}x\right)\psi_{j}(x)dx\label{eq-w-1}\\
	\frac{d}{dt}\underset{j\rightarrow\infty}{\lim}\int_{\mathbb{R}^{n}}\phi_{j}\left(e^{-tA}x\right)\psi_{j}(x)dx=\underset{j\rightarrow\infty}{\lim}\frac{d}{dt}\int_{\mathbb{R}^{n}}\phi_{j}\left(e^{-tA}x\right)\psi_{j}(x)dx\label{eq-w-2}\\
	\frac{d}{dt}\int_{\mathbb{R}^{n}}\phi_{j}\left(e^{-tA}x\right)\psi_{j}(x)dx=\int_{\mathbb{R}^{n}}\left\langle \nabla\phi_{j}(x),-\psi_{j}(e^{tA}x)Ax\right\rangle dx\label{eq-w-3}\\
	\underset{j\rightarrow\infty}{\lim}\int_{\mathbb{R}^{n}}\left\langle \nabla\phi_{j}(x),-\psi_{j}(e^{tA}x)Ax\right\rangle dx=\int_{S^{n-1}\cap e^{-tA}K}\left\langle x,Ax\right\rangle dS\label{eq-w-4}
	\end{eqnarray}
	The equality \eqref{eq-w-1} is a direct consequence of Lebesgue's
	dominated convergence theorem. For \eqref{eq-w-3}, note that 
		\[
		\frac{d}{dt}\phi_{j}\left(e^{-tA}x\right)\psi_{j}(x)dx=\left\langle \nabla\phi_{j}(e^{-tA}x),-\psi_{j}(x)Ae^{-tA}x\right\rangle 
		\]
and that by Leibniz's integral rule,  \[\frac{d}{dt}\int_{\mathbb{R}^{n}}\phi_{j}\left(e^{-tA}x\right)\psi_{j}(x)dx=\int_{\mathbb{R}^{n}}\frac{d}{dt}\phi_{j}\left(e^{-tA}x\right)\psi_{j}(x)dx.\] 
It follows that for every fixed $j\in\mathbb{N}$ (recall $trA=0$),
	\begin{eqnarray*}
		\frac{d}{dt}\int_{\mathbb{R}^{n}}\phi_{j}\left(e^{-tA}x\right)\psi_{j}(x)dx & = & \int_{\mathbb{R}^{n}}\left\langle \nabla\phi_{j}(e^{-tA}x),-\psi_{j}(x)Ae^{-tA}x\right\rangle dx\\
		& = & \int_{\mathbb{R}^{n}}\left\langle \nabla\phi_{j}(x),-\psi_{j}(e^{tA}x)Ax\right\rangle dx,
	\end{eqnarray*}
	proving \eqref{eq-w-3}.
	
	To prove \eqref{eq-w-2} and \eqref{eq-w-4}, it is enough to show the following:
	
	\begin{claim}
		\label{claim-onemoreclaim} There is a neighborhood of $t=0$ where the function $\frac{d}{dt}f_{j}(t)=\frac{d}{dt}\int_{\mathbb{R}^{n}}\phi_{j}\left(e^{-tA}x\right)\psi_{j}(x)dx$
		converges uniformly to $\int_{S^{n-1}\cap e^{-tA}K}\left\langle x,Ax\right\rangle dS$. 
	\end{claim}
	
\noindent	\emph{Proof.}
		Denote 
		\[
		M_j:=\left\{ x:1-\frac{1}{j}\leq|x|\leq1\right\} \supset\text{supp}\nabla\phi_{j}(x).
		\]
		Then 
		\[
		\int_{\mathbb{R}^{n}}\left\langle \nabla\phi_{j}(x),-\psi_{j}(e^{tA}x)Ax\right\rangle dx=\int_{M_j}\left\langle \nabla\phi_{j}(x),-\psi_{j}(e^{tA}x)Ax\right\rangle dx.
		\]
		The functions $\phi_{j}(x),-\psi_{j}(e^{tA}x)Ax$ are continuously
		differentiable, $\partial M_j$ is smooth, and so we may integrate
		by parts to have 
		\begin{eqnarray}
		&  & \int_{M_j}\left\langle \nabla\phi_{j}(x),-\psi_{j}(e^{tA}x)Ax\right\rangle dx=\label{eq-some-eq}\\
		&  & =\int_{\partial M_j}\phi_{j}\left(x\right)\left\langle \vec{n},-\psi_{j}(e^{tA}x)Ax\right\rangle dS+\int_{M_j}\phi_{j}\left(x\right)\mbox{div}(\psi_{j}(e^{tA}x)Ax)dx\nonumber 
		\end{eqnarray}
		where $\vec{n}$ is the outward unit normal of $M_j$. Note that 
		\[
		\mbox{div}(\psi_{j}(e^{tA}x)Ax)=\left\langle \nabla\psi_{j}(e^{tA}x),Ax\right\rangle +\psi_{j}(e^{tA}x)\mbox{div}Ax=
		\]
		\[
		=\left\langle \nabla\psi_{j}(e^{tA}x),Ax\right\rangle +\psi_{j}(e^{tA}x)\text{tr}A=\left\langle \nabla\psi_{j}(e^{tA}x),Ax\right\rangle 
		\]
		and so 
		\[
		\left|\int_{M(j)}\phi_{j}\left(x\right)\mbox{div}(\psi_{j}(e^{tA}x)Ax)dx\right|\leq\int_{M(j)}\left|\left\langle \nabla\psi_{j}(e^{tA}x),Ax\right\rangle \right|dx.
		\]
		There is a constant $c$ such that 
		\[
		\int_{M_j}\left|\left\langle \nabla\psi_{j}(e^{tA}x),Ax\right\rangle \right|dx\leq c\sqrt{j}\mbox{Vol}_{n}(M(j))=c\kappa_{n}\sqrt{j}\left(1-\left(\frac{j-1}{j}\right)^{n}\right)\underset{j\rightarrow\infty}{\rightarrow}0.
		\]
		Going back to \eqref{eq-some-eq}, we have shown that $\int_{M_j}\phi_{j}\left(x\right)\mbox{div}(\psi_{j}(e^{tA}x)Ax)dx$
		converges uniformly to $0$. As for $\int_{\partial M_j}\phi_{j}\left(x\right)\left\langle \vec{n},-\psi_{j}(e^{tA}x)Ax\right\rangle dS$,
		note that 
		\[
		\partial M_j=S^{n-1}\cup\frac{j-1}{j}S^{n-1}
		\]
		where $\phi_{j}\left(x\right)=0$ on $S^{n-1}$, and $\phi_{j}\left(x\right)=1$
		on $\frac{j-1}{j}S^{n-1}$. For every $x\in\frac{j-1}{j}S^{n-1}$,
		the outer unit normal $\vec{n}$ of $M(j)$ is $-\frac{j}{j-1}x$,
		and so: 
		\begin{eqnarray*}
			\int_{\partial M_j}\phi_{j}\left(x\right)\left\langle \vec{n},-\psi_{j}(e^{tA}x)Ax\right\rangle dS & = & \int_{\frac{j-1}{j}S^{n-1}}\psi_{j}(e^{tA}x)\left\langle \vec{n},-Ax\right\rangle dS\\
			=\frac{j}{j-1}\int_{\frac{j-1}{j}S^{n-1}}\psi_{j}(e^{tA}x)\left\langle x,Ax\right\rangle dS & = & \left(\frac{j-1}{j}\right)^{n}\int_{S^{n-1}}\psi_{j}\left(\frac{j-1}{j}e^{tA}x\right)\left\langle x,Ax\right\rangle dS.
		\end{eqnarray*}
		We will show that there is some sequence $\xi(j)\rightarrow0$ and
		some $\delta>0$ such that for every $\left|t\right|<\delta$, 
		\[
		\left|\int_{S^{n-1}}\left(\frac{j-1}{j}\right)^{n}\psi_{j}\left(\frac{j-1}{j}e^{tA}x\right)\left\langle x,Ax\right\rangle dS-\int_{S^{n-1}}\mathbf{1}_{K}(e^{tA}x)\left\langle x,Ax\right\rangle dS\right|\leq\xi(j).
		\]
		
		Denote $\nu_{j}(x)=\left(\frac{j-1}{j}\right)^{n}\psi_{j}\left(\frac{j-1}{j}x\right)$,
		and consider 
		\begin{eqnarray*}
			\left|\int_{S^{n-1}}\nu_{j}(e^{tA}x)\left\langle x,Ax\right\rangle dS-\int_{S^{n-1}}\mathbf{1}_{K}(e^{tA}x)\left\langle x,Ax\right\rangle dS\right|\leq\\
			\leq c\int_{S^{n-1}}\left|\nu_{j}(e^{tA}x)-\mathbf{1}_{K}(e^{tA}x)\right|dS
		\end{eqnarray*}
		The set $S^{n-1}$ is a union of the following three sets: 
		\begin{eqnarray*}
			S_{1}(j,t) & = & \left\{ x\in S^{n-1}:\left\Vert e^{tA}x\right\Vert _{K}\geq\frac{j}{j-1}\right\} \\
			S_{2}(j,t) & = & \left\{ x\in S^{n-1}:\left\Vert e^{tA}x\right\Vert _{K}\leq1-\frac{1}{\sqrt{j}}\right\} \\
			S_{3}(j,t) & = & \left\{ x\in S^{n-1}:\frac{\sqrt{j}-1}{\sqrt{j}}\leq\left\Vert e^{tA}x\right\Vert _{K}\leq\frac{j}{j-1}\right\} .
		\end{eqnarray*}
		On $S_{1}$ we have that $\nu_{j}(x)=\mathbf{1}_{K}(e^{tA}x)=0$ and
		so $\int_{S_{1}(j,t)}\left|\nu_{j}(e^{tA}x)-\mathbf{1}_{K}(e^{tA}x)\right|dS=0$
		for all $j,t$. On $S_{2}$ we have that $\nu_{j}(e^{tA}x)=\left(\frac{j-1}{j}\right)^{n},\mathbf{1}_{K}(e^{tA}x)=1$
		and so: 
		\[
		\int_{S_{2}(j,t)}\left|\nu_{j}(e^{tA}x)-\mathbf{1}_{K}(e^{tA}x)\right|dS=\left|\left(\frac{j-1}{j}\right)^{n}-1\right|\mbox{Vol}_{n-1}(S_{2}(j,t)).
		\]
		There is a constant $c$ such that $\mbox{Vol}_{n-1}(S_{2}(j,t))\leq c$
		for all $j\in\mathbb{N}$ and for all $t\in[1,-1]$. It follows that
		\[
		\left|\left(\frac{j-1}{j}\right)^{n}-1\right|\mbox{Vol}_{n-1}(S_{2}(j,t))\leq\left|\left(\frac{j-1}{j}\right)^{n}-1\right|c\underset{j\to\infty}{\to}0.
		\]
		Finally, on $S_{3}$ we have that 
		\[
		\int_{S_{3}(j,t)}\left|\nu_{j}(e^{tA}x)-\mathbf{1}_{K}(e^{tA}x)\right|dS\leq\xi_{j}(t)
		\]
		where 
		\[
		\xi_{j}(t)=\vol_{n-1}\left(S_{3}(j,t)\right)=\mbox{Vol}_{n-1}\left\{ x\in S^{n-1}:\frac{\sqrt{j}-1}{\sqrt{j}}\leq\left\Vert e^{tA}x\right\Vert _{K}\leq\frac{j}{j-1}\right\} 
		\]
		is monotonically decreasing in $j$ for every fixed $t$. By Dini's
		theorem, $\xi_{j}(t)$ converges uniformly
		to $\xi(t)=\mbox{Vol}_{n-1}\left(S^{n-1}\cap\partial e^{tA}K\right)$.
		Assuming $\mbox{Vol}_{n-1}\left(S^{n-1}\cap\partial K\right)=0$ and
		$\mbox{Vol}_{n-1}\left(\partial\mathcal{E}\cap\partial K\right)=0$
		for all but finitely many ellipsoids, there is some $\delta>0$ such
		that $\xi(t)=\mbox{Vol}_{n-1}\left(S^{n-1}\cap\partial e^{tA}K\right)=0$
		for all $\left|t\right|<\delta$, and so on the set $|t|<\delta$,
		the sequence $\int_{S_{3}}\left|\nu_{j}(e^{tA}x)-\mathbf{1}_{K}(e^{tA}x)\right|dS$
		converges uniformly to $0$. This proves Claim \ref{claim-onemoreclaim}
		and with it Theorem \ref{thm:derivative-of-volume}. 
\end{proof}

\begin{rem}
	Note that the proof above shows that the conditions of Theorem \ref{thm:derivative-of-volume} (and therefore of Theorem \ref{thm:max-intersection-isotropic})
	may be slightly relaxed: in fact, we do not need $\vol_{n-1}(K\cap\mathcal{E})=0$
	for all but finitely many ellipsoids. It is enough to have a neighborhood
	$N\subset SL_{n}$ of $I_{n}$ such that $\vol_{n-1}(K\cap T\mathcal{E})=0$
	for all $T\in N$. 
\end{rem}

	\section{Proof of the main theorems}
	\label{sec-main-proofs}

In this section we use the results of Section \ref{sec-prelim} to provide short proofs to the three main Theorems \ref{thm:max-intersection-isotropic}, \ref{thm:measure-limit-john}, and \ref{thm:measure-limit-loewner}.

As we mentioned, the proof of Theorem \ref{thm:max-intersection-isotropic} follows almost directly from Theorem \ref{thm:derivative-of-volume}:
	
	\begin{proof}[Proof of Theorem \ref{thm:max-intersection-isotropic}]
		First note that $K$ is in maximal intersection position of radius $r$ if and only if $r^{-1}K$ is in maximal intersection position of radius $1$, and so it is enough to prove the theorem in the case $r=1$.
		
		Let $W:SL_{n}\rightarrow\mathbb{R}$, $W(T)=\mbox{Vol}_{n}(K\cap TB_{2}^{n})$.
		If $I_{n}$ is a local maximum of $W$, then for any $A\in M_{n}(\mathbb{R})$
		such that $\text{tr}A=0$, the derivative $\left.\frac{dW(e^{tA})}{dt}\right|_{t=0}=\left.\frac{dV(t)}{dt}\right|_{t=0}$
		is either zero or does not exist. Theorem \ref{thm:derivative-of-volume}
		states that the derivative does exist for all $A$, and it equals
		$\int_{S^{n-1}\cap K}\left\langle x,Ax\right\rangle dS(x)$.
		It follows that 
		\[\int_{S^{n-1}}\left\langle x,Ax\right\rangle d\mu_{K}=\frac{1}{\vol_{n-1}(S^{n-1}\cap K)}\int_{S^{n-1}\cap K}\left\langle x,Ax\right\rangle dS=0\]
		for all $A$ such that $\mbox{tr}A=0$, and by Lemma \ref{lem:isotropic-traceless-matrix}, $\mu_{K}$ is isotropic. 
	\end{proof}
	
	As we have mentioned, the result of Theorem \ref{thm:max-intersection-isotropic} resembles that of John's Theorem (Theorem \ref{thm:john}), but does not include it. However, Theorem \ref{thm:max-intersection-isotropic} provides a family of isotropic measures which are used in the proof of Theorem \ref{thm:measure-limit-john}:
	
	\begin{proof}[Proof of Theorem \ref{thm:measure-limit-john}]
		Let $r\searrow1$. By Lemma \ref{lem:max-ellipsoid-exists}, we may choose an intersection maximizing ellipsoid
		$\mathcal{E}_{r}$ for each $r$. By Lemma \ref{lem-convergence-of-ellipse-toJohn},
		$\mathcal{E}_{r}\to B_{2}^{n}$ and so we may choose a sequence of
		positive definite transformations $T_{r}\rightarrow I_{n}$ such that
		$B_{2}^{n}=T_{r}\mathcal{E}_{r}$. Then $T_{r}K$ is in maximal intersection
		position of radius $r$ and $\vol_{n-1}(\partial T_{r}K\cap S^{n-1})=0$
		for almost all $r$. By Theorem \ref{thm:max-intersection-isotropic}, the probability measures
		on the sphere 
		\[
		\mu_{r}(A)=\mu_{S^{n-1}\backslash T_{r}K}(A)=\frac{\sigma\left(A\backslash T_{r}K\right)}{\sigma\left(S^{n-1}\backslash T_{r}K\right)}
		\]
		are isotropic.
		
		Note that $S^{n-1}$ is a compact metric space, and so the family of measures $\mu_{r}$
		has a weakly converging subsequence $\mu_{j}\rightarrow\mu$ where
		$\mu$ is a probability measure on $S^{n-1}$. We will show that the limit measure $\mu$ is an isotropic measure
		whose support lies in $\partial K\cap S^{n-1}$.
		
		First, weak convergence implies 
		\[
		\int_{S^{n-1}}\left\langle x,\theta\right\rangle ^{2}d\mu_{j}(x)\rightarrow\int_{S^{n-1}}\left\langle x,\theta\right\rangle ^{2}d\mu(x)
		\]
		and 
		\[
		\frac{1}{n}=\frac{\mu_{j}(S^{n-1})}{n}\rightarrow\frac{\mu(S^{n-1})}{n}
		\]
		so $\int_{S^{n-1}}\left\langle x,\theta\right\rangle ^{2}d\mu(x)=\frac{\mu(S^{n-1})}{n}=\frac{1}{n}$,
		i.e., $\mu$ is isotropic.
		
		Second, let 
		\[
		U_{k}=\left\{ x\in S^{n-1}:d(x,\partial K)>\frac{1}{k}\right\} 
		\]
		where $d(\cdot,\cdot)$ is a metric on $S^{n-1}$. The measure $\mu_{j}$
		is supported on $S^{n-1}\backslash T_{r_{j}}K$ where $T_{r_{j}}K\rightarrow K$,
		and so there is $M$ such that for any $k>M$ there is some $N(k)$
		such that $\mu_{j}(U_{k})=0$ for all $j>N(k)$. Since $U_{k}$ is
		open, weak convergence implies $\mu(U_{k})\leq\liminf\mu_{j}(U_{k})=0$,
		so $\mu(U_{k})=0$ for all $k>M$. It follows that $\mu\left(\bigcup_{k=M}^{\infty}U_{k}\right)=\underset{k\rightarrow\infty}{\lim}\mu(U_{k})=0$,
		where 
		\[
		\bigcup_{k=M}^{\infty}U_{k}=\{x\in S^{n-1}:d(x,\partial K)>0\}=S^{n-1}\backslash\mbox{cl}\partial K=S^{n-1}\backslash\partial K.
		\]
		It follows that $\mu(S^{n-1}\backslash\partial K)=0$ and so $\text{supp}\mu\subset S^{n-1}\cap\partial K$. 
	\end{proof}

	The proof of Theorem \ref{thm:measure-limit-loewner} is analogous to that of Theorem \ref{thm:measure-limit-john},
	only here we use 
	\[
	\nu_{j}(A)=\mu_{T_{r_{j}}K}(A)=\frac{\sigma\left(A\cap T_{r_{j}}K\right)}{\sigma\left(S^{n-1}\cap T_{r_{j}}K\right)}
	\]
	which is isotropic by Theorem \ref{thm:max-intersection-isotropic}. In this
	case, it is the measures $\nu_{j}$ that satisfy $\nu_{j}(U_{k})=0$
	for all $j>N(k)$, rather than the measures $\mu_{j}$ . In other words,
	for a John-type measure we use a sequence of uniform measures ``outside''
	$T_{r_{j}}K$, whereas for a Loewner-type measure we use a sequence
	of uniform measures ``inside'' $T_{r_{j}}K$.

	\section{Remarks about uniqueness following from the (B) property}
	\label{section-uniquenessness}
	
	Throughout this text we discussed maximal intersection positions of a body $K$. While Lemma \ref{lem:max-ellipsoid-exists} 
	shows that such a position always exists, we did not show that this measure is unique. 
	If $0<r<r_{J}$ or $r>r_{L}$ then the maximum intersection ellipsoid $\mathcal{E}_{r}$ of radius $r$ is clearly not unique.
	If $r=r_{J}$ or $r=r_{L}$ then $\mathcal{E}_{r}$ is unique, by John's theorem. 
	The question of uniqueness remains open for the case $r_{J}<r<r_{L}$,
	but it is implied by a variant of a well known conjecture which we next discuss:
	\begin{conjecture}
		\label{conj-strongBwEq} For a convex body $K\subset\mathbb{R}^{n}$
		and a diagonal $n\times n$ matrix $\Lambda$, the function 
		\[
		\phi(t)=\vol_{n}\left(e^{t\Lambda}K\cap B_{2}^{n}\right)
		\]
		is log-concave in $t$, i.e. 
		\begin{equation}
		\vol_{n}\left(e^{\frac{t}{2}\Lambda}K\cap B_{2}^{n}\right)^{2}\geq\vol_{n}\left(e^{t\Lambda}K\cap B_{2}^{n}\right)\vol_{n}\left(K\cap B_{2}^{n}\right)\label{eq-Bconjeq}
		\end{equation}
		for all $t\in\mathbb{R}$ and all diagonal $\Lambda$. Furthermore,
		equality is attained if and only if one of the following hold: $K\subset B_{2}^{n}$,
		$B_{2}^{n}\subset K$, or $\Lambda=\lambda I_{n}$ for some $\lambda\in\mathbb{R}$. 
	\end{conjecture}
	\begin{prop}
		Assuming Conjecture \ref{conj-strongBwEq} is true, if $K$ is a centrally
		symmetric convex body, the maximum intersection ellipsoid of radius $r$ is unique for
		$r_{J}<r<r_{L}$. 
	\end{prop}
	\begin{proof}
		Letting $r_{J}<r<r_{L}$, assume there are two distinct maximum intersection ellipsoid of radius $r$. We may assume that one of these ellipsoids is $B_{2}^{n}$,
		and the other is of the form $e^{\Lambda}B_{2}^{n}$ where $\Lambda$
		is a diagonal matrix with $\text{tr}\Lambda=0$. Conjecture \ref{conj-strongBwEq}
		now gives 
		\[
		\vol_{n}\left(K\cap e^{\frac{\Lambda}{2}}B_{2}^{n}\right)\geq\vol_{n}\left(K\cap B_{2}^{n}\right)
		\]
		where maximality of $B_{2}^{n}$ implies equality in the above. Since
		$r_{J}<r<r_{L}$, we have $K\nsubseteq B_{2}^{n}$ and $B_{2}^{n}\nsubseteq K$.
		It follows that $\Lambda$ is a traceless scalar matrix, i.e. $\Lambda$
		is the zero matrix and $e^{\Lambda}=I_{n}$. 
	\end{proof}
	Conjecture \ref{conj-strongBwEq} describes a (B)-type property on
	the Lebesgue measure on $B_{2}^{n}$, under the following terminology: 
	\begin{defn}
		Given a measure $\mu$ on $\mathbb{R}^{n}$ and a measurable set $K\subset\mathbb{R}^{n}$,
		we say that $\mu$ and $K$ have \emph{the} \emph{weak $(B)$ property}
		if the function 
		\[
		t\mapsto\mu(e^{t}K)
		\]
		is log-concave on $\mathbb{R}$. \\
		Denoting $\mbox{diag}(t_{1},...,t_{n})$ the diagonal matrix with
		diagonal entries $t_{1},...,t_{n}$, we will say that $\mu$ and $K$
		have \emph{the strong (B) property} if the function 
		\[
		(t_{1},...,t_{n})\mapsto\mu(e^{\mbox{diag}(t_{1},...,t_{n})}K)
		\]
		is log-concave on $\mathbb{R}^{n}$. 
	\end{defn}
	The notion of the (B) property arises from a problem proposed by Banaszczyk
	and described by Latala \cite{latala2003some} known as the (B) conjecture
	(now the (B) theorem), where, in the terminology as above, it was
	conjectured that the standard Gaussian probability measure $\gamma$
	on $\mathbb{R}^{n}$ and any centrally symmetric convex body $K\subset\mathbb{R}^{n}$
	have the weak (B) property. The (B) conjecture was solved by Cordero-Erausquin,
	Fradelizi, and Maurey \cite{cordero2004b}, where it was shown that
	$\gamma$ and $K$ have in fact a strong (B) property.
	
	Conjecture \ref{conj-strongBwEq} proposes that the uniform Lebesgue
	measure on $B_{2}^{n}$ and any centrally symmetric convex body have
	the strong (B) property, with further assumptions on the equality
	case.
	
	Unfortunately not a lot is known about the (B) property of general
	measures, and even less about the equality case. We will briefly mention
	what is currently known: Livne Bar-on \cite{bar2014b} showed that in $\mathbb{R}^{2}$,
	the uniform Lebesgue measure on a centrally symmetric convex body
	$L\subset\mathbb{R}^{2}$ has the weak (B) property with any centrally
	symmetric convex body $K\subset\mathbb{R}^{2}$. This result was generalized
	by Saroglou \cite{saroglou2015remarks}, where it was shown that if
	the log-Brunn-Minkowski inequality holds in dimension $n$, then the
	uniform probability measure on the $n-$dimensional cube has the strong
	(B) property, and the uniform probability measure of every centrally
	symmetric convex body has the weak (B) property, with any centrally
	symmetric convex body $K$.
	
	The log-Brunn-Minkowski inequality states that for two centrally symmetric
	convex bodies ${K,L\subset\mathbb{R}^{n}}$ and $\lambda\in[0,1]$,
	\begin{equation}
	\vol_{n}\left(\left(1-\lambda\right)K+_{o}\lambda L\right)\geq\vol_{n}(K)^{1-\lambda}\vol_{n}(L)^{\lambda}\label{eq-logBM}
	\end{equation}
	where 
	\[
	(1-\lambda)K+_{o}\lambda L=\bigcap_{u\in S^{n-1}}\left\{ x:\left\langle x,u\right\rangle \leq h_{K}(u)^{1-\lambda}h_{L}(u)^{\lambda}\right\} .
	\]
	
	It was shown by Böröczky, Lutwak, Yang, and Zhang \cite{boroczky2012log}
	that the log-Brunn-Minkowski inequality holds for $n=2$, and so together
	with \cite{saroglou2015remarks} the result of \cite{bar2014b} is
	implied.
	
	In a recent publication \cite{Saroglou2016}, Saroglou states that
	an unconditional log-concave measure $\mu$ and an unconditional body
	$K$ have the strong (B) property. For our purposes, it is enough
	to mention that the uniform measure on $B_{2}^{n}$ is unconditional
	log-concave. It follows that Conjecture \ref{conj-strongBwEq} (without
	the equality case) holds whenever $K$ is unconditional, i.e. $(x_{1},\ldots,x_{n})\in K$
	implies $(\delta_{1}x_{1},\ldots,\delta_{n}x_{n})\in K$ for any choice
	of $\delta_{i}\in\{-1,1\}$ where $i=1,...,n$.
	
	Still not a lot is known on equality cases in inequalities such as
	\eqref{eq-Bconjeq}. In \cite{Saroglou2016}, Saroglou expands further
	on the relationship between the (B) property and the log-Brunn-Minkowski,
	and conjectures that equality in \eqref{eq-logBM} is attained if
	and only if $K=K_{1}\times\ldots\times K_{m}$ for some convex sets
	$K_{1},\ldots K_{m}$ that cannot be written as cartesian products
	of lower dimensional sets, and $L=c_{1}K_{1}\times\ldots c_{m}K_{m}$
	for some positive numbers $c_{1},\ldots,c_{m}$. We have not found
	similar conjectures or results regarding the equality case in \eqref{eq-Bconjeq}.

	\bibliographystyle{plain}
	\bibliography{sources}

\begin{thebibliography}{10}
	
	\bibitem{Artstein-Avidan2015}
	S.~Artstein-Avidan, A.~Giannopoulos, and V.D. Milman.
	\newblock {\em Asymptotic Geometric Analysis, Part {I}}, volume 202.
	\newblock American Mathematical Soc., 2015.
	
	\bibitem{Ball1992}
	K.~Ball.
	\newblock Ellipsoids of maximal volume in convex bodies.
	\newblock {\em Geometriae Dedicata}, 41(2):241--250, 1992.
	
	\bibitem{boroczky2012log}
	K.J. B{\"o}r{\"o}czky, E.~Lutwak, D.~Yang, and G.~Zhang.
	\newblock The log-{B}runn-{M}inkowski inequality.
	\newblock {\em Advances in Mathematics}, 231(3):1974--1997, 2012.
	
	\bibitem{cordero2004b}
	D.~Cordero-Erausquin, M.~Fradelizi, and B.~Maurey.
	\newblock The {(B)} conjecture for the {G}aussian measure of dilates of
	symmetric convex sets and related problems.
	\newblock {\em Journal of Functional Analysis}, 214(2):410--427, 2004.
	
	\bibitem{Giannopoulos2000}
	A.~Giannopoulos and V.D. Milman.
	\newblock Extremal problems and isotropic positions of convex bodies.
	\newblock {\em Israel Journal of Mathematics}, 117(1):29--60, 2000.
	
	\bibitem{john1948extremum}
	F.~John.
	\newblock Extremum problems with inequalities as side constraints.
	\newblock {\em Studies and Essays, Courant Anniversary Volume}, pages 187--204,
	1948.
	
	\bibitem{latala2003some}
	R.~Lata{\l}a.
	\newblock On some inequalities for {G}aussian measures.
	\newblock {\em Proceedings of the International Congress of Mathematicians,
		Beijing}, II, 2002.
	
	\bibitem{bar2014b}
	A.~Livne Bar-on.
	\newblock The {(B)} conjecture for uniform measures in the plane.
	\newblock In {\em Geometric Aspects of Functional Analysis}, pages 341--353.
	Springer, 2014.
	
	\bibitem{milman1986inegalite}
	V.D. Milman.
	\newblock In{\'e}galit{\'e} de {B}runn-{M}inkowski inverse et applicationsa la
	th{\'e}orie locale des espaces norm{\'e}s.
	\newblock {\em CR Acad. Sci. Paris}, 302(1):25--28, 1986.
	
	\bibitem{saroglou2015remarks}
	C.~Saroglou.
	\newblock Remarks on the conjectured log-{B}runn-{M}inkowski inequality.
	\newblock {\em Geometriae Dedicata}, 177(1):353--365, 2015.
	
	\bibitem{Saroglou2016}
	C.~Saroglou.
	\newblock More on logarithmic sums of convex bodies.
	\newblock {\em Mathematika}, 62(03):818--841, 2016.
	
\end{thebibliography}
	
\end{document}